\documentclass[11pt]{amsart}
\usepackage{amsmath,amsfonts,latexsym,graphicx,amssymb,url}
\usepackage{amsmath,txfonts,pifont,bbding,pxfonts,manfnt}
\usepackage{psfrag}
\usepackage{wrapfig, subfigure,graphicx}
\usepackage[colorlinks=true, citecolor=blue, linkcolor=blue]{hyperref}

\usepackage[active]{srcltx}
\usepackage{pdfsync}
\usepackage{graphicx}
\usepackage[usenames,dvipsnames]{xcolor}
\usepackage{amsmath}
\usepackage{amsfonts}
\usepackage{amssymb}

\newcommand{\R}{{\mathbb R}} %%reals

\renewcommand{\(}{\left(}
\renewcommand{\)}{\right)}

\def \e {\varepsilon}

\def\a{\alpha}
\def\o{\omega}
\newtheorem{theorem}{Theorem}[section]

\newtheorem{lemma}[theorem]{Lemma}

\newtheorem{definition}[theorem]{Definition}
\newtheorem{remark}[theorem]{Remark}

\begin{document}

\title[$L^\infty$-estimates for non quadratic costs, \today]
{$L^\infty$-estimates in optimal transport for non quadratic costs\\ \today}
\author[C. E. Guti\'errez]
{Cristian E. Guti\'errez}
\author[A. Montanari]
{Annamaria Montanari}
\thanks{\today
}
\address{Department of Mathematics\\Temple University\\Philadelphia, PA 19122}
\email{cristian.gutierrez@temple.edu}
\address{Dipartimento di Matematica\\Piazza di Porta San Donato 5\\Universit\`a di Bologna\\40126 Bologna, Italy}
\email{annamaria.montanari@unibo.it}

\maketitle

\begin{abstract}
For cost functions $c(x,y)=h(x-y)$ with $h\in C^2$ homogeneous of degree $p\geq 2$, we show $L^\infty$-estimates of $Tx-x$ on balls, where $T$ is an $h$-monotone map. Estimates for the interpolating mappings $T_t=t(T-I)+I$ are deduced from this. 
\end{abstract}
\tableofcontents

\setcounter{equation}{0}

\section{Introduction}
\setcounter{equation}{0}
This note originates looking into the recent and very interesting paper by M. Goldman and F. Otto \cite{2020-goldman-otto-variational} containing a new proof of the regularity of optimal maps for the Monge problem when the cost is quadratic. 
Our intention has been to investigate the validity of similar results for powers costs  $|x-y|^p$ with $p\geq 2$, and
in that endeavor we came up with local $L^\infty$-estimates for monotone and interpolating maps relative to that cost, inequalities \eqref{eq:main L infinity estimate in a general ball} and \eqref{eq:inclusion T_t^-1 ball contained into a ball}, respectively; these extend \cite[Lemma 3.1]{2020-goldman-otto-variational}. More generally, our estimates hold when the cost is given by a $C^2$ function that is homogeneous of degree $p$.
Since we believe that these estimates may be useful to obtain regularity results for optimal transport when $p\neq 2$, and may have independent interest, it is our purpose to present them here. Moreover, we are able to show that these estimates suffice to prove, with modifications, several important steps in parallel with those carried out in \cite{2020-goldman-otto-variational} toward the super-linear growth as in Prop. 3.3, eq. (3.15) of that paper; we will not provide these details in this note. However, a missing part is a replacement for $p\neq 2$ of the so called quasi-orthogonality property proved in \cite[Step 3, proof of Prop. 3.3]{2020-goldman-otto-variational}.
Recent regularity results for general cost functions are considered in \cite{2020-otto-prodhomme-ried-general-costs} but they do not include the case of non quadratic power costs, see Remark \ref{rmk:for pneq 2 C_4 does not hold}. 
We mention that global $L^\infty$ estimates for optimal maps in terms of the $p$-Wasserstein distance are proved in \cite{2007-bouchitte}.

The note is organized as follows. Section \ref{sec:L-infty-estimates} contains a detailed proof of the $L^\infty$-estimate \eqref{eq:main L infinity estimate in a general ball} on general balls.
In Section \ref{sec:Estimates for the displacement interpolating map}, we introduce a notion of monotonicity \eqref{eq:A} that is equivalent to \eqref{eq:map T is c-monotone h}
and used it to prove in Section \ref{sec:L-infty-estimates of the interpolating map} the estimate \eqref{eq:inclusion T_t^-1 ball contained into a ball} for interpolating maps.
Section \ref{sec:L-infty-estimates of densities} shows, as a consequence, $L^\infty$-estimates for the densities of the transport problem. Section \ref{sec:fluids} shows that the quantity on the right hand side of the $L^\infty$-estimate \eqref{eq:main L infinity estimate in a general ball} is comparable to an integral of a fluid flow. 
Section \ref{sec:diff monotone maps} is self-contained and shows an $L^\infty$-estimate for monotone maps minus an arbitrary affine function,
Lemma \ref{lm:estimates of T with A}, which implies point-wise differentiability of locally integrable monotone maps, see Theorem \ref{thm:differentiability of monotone maps} and Remark \ref{eq:ambrosio result bounded deformation}. 
Finally and for convenience, we include an appendix with the known formula \eqref{eq:third Green identity} which is the starting point to prove the main estimate in Section \ref{sec:L-infty-estimates}.

{\bf Acknowledgements.}
We would like to thank Craig Evans for useful comments and for pointing out Krylov's work \cite{K83}; see Remark \ref{rmk:mignot ambrosio and krylov results}. And we like to thank also Luigi Ambrosio for pointing out the connection between monotone maps and maps of bounded deformation, Remark \ref{eq:ambrosio result bounded deformation}, and useful comments. 
C.E.G was partially supported
by NSF grant DMS--1600578, and A.M. was partially supported by  
a grant from GNAMPA of INdAM.

\section{$L^\infty$-estimates}\label{sec:L-infty-estimates}
If $c(x,y):D\times D^*\to [0,+\infty)$ is a general cost function, then 
from optimal transport theory, the optimal map for the Monge problem is given by $T=\mathcal N_{c,\phi}$ where $\phi$ is $c$-concave and 
\[
\mathcal N_{c,\phi}(x)
=
\left\{m\in D^*:\phi(x)+\phi^c(m)=c(x,m)\right\}
\]
with $\phi^c(m)=\inf_{x\in D}\(c(x,m)-\phi(x)\)$, see for example \cite[Sect. 3.2]{gutierrez-huang:farfieldrefractor}.
This implies that
\begin{equation}\label{eq:c-monotonicty}
c\(x,Tx\)+c\(y,Ty\)\leq c\(x,Ty\)+c\(y,Tx\)
\end{equation}
assuming $Tx$ is single valued for a.e. $x\in D$.
In our analysis below we will only use that $T$ satisfies \eqref{eq:c-monotonicty}; and {\it that $T$ is optimal will not be used}.

We assume that the cost $c$ has the form $c(x,y)=h(x-y)$ where {\it $h\geq 0$ is a $C^2$ convex function in $\R^n$}. 
What we have in mind is to obtain $L^\infty$-estimates for $u(x)=Tx-x$, as in the paper by Goldman and Otto \cite[Lemma 3.1]{2020-goldman-otto-variational}, but when {\it $h$ is positively homogenous of degree $p$ for some $1<p<\infty$}.
For this $c$, \eqref{eq:c-monotonicty} obviously reads
\begin{equation}\label{eq:map T is c-monotone h}
h\(x-Tx\)+ h\(y-Ty\)\leq h\(x-Ty\)+ h\(y-Tx\),
\end{equation}
that is, $T$ is $h$-monotone, or equivalently
\begin{equation}\label{eq:c-monotonicity h}
h\(-u(x)\)+ h\(-u(y)\)\leq h\(x-y-u(y)\)+ h\(y-x-u(x)\).
\end{equation}
Defining 
\[
G(a,b)=h(a-b)-h(a)-h(b),
\]
and assuming that {\it $h$ is even,}
the inequality \eqref{eq:c-monotonicity h} reads
\begin{equation}\label{eq:inequality between G} 
-G\(x-y,u(y)\)\leq G\(y-x,u(x)\)+2\,h(x-y).
\end{equation}

Our purpose is then to prove the following local $L^\infty$-estimate.
\begin{theorem}\label{thm:main Linfty estimate}
Suppose $h\in C^2(\R^n)$ is nonnegative, even, convex, positively homogeneous of degree $p$, for some $p\geq 2$, and $\min_{x\in S^{n-1}} h(x)=m>0$. If $T$ is a map satisfying the monotonicity condition \eqref{eq:map T is c-monotone h} for a.e. $x,y\in \R^n$ and $u(x)=Tx-x$, then      
{\small \begin{equation}\label{eq:main L infinity estimate in a general ball}
\sup_{y\in B_{\beta\,R}(x_0)}|u(y)|
\leq
\begin{cases}
L_1\,R^{n/(n+p)}\,\(\fint_{B_R(x_0)}|u(x)|^p\,dx\)^{1/(n+p)} & 
\text{if $\frac{1}{R^p}\,\fint_{B_R(x_0)}|u(x)|^p\,dx\leq \(\dfrac{1-\beta}{2}\)^{n+p} \dfrac{(p-1)\,C_2}{(n+1)\,C_1\,\omega_n}$}\\
L_2\,\(R^{-1}\,\fint_{B_R(x_0)}|u(x)|^p\,dx\)^{1/(p-1)}  & \text{if $\frac{1}{R^p}\,\fint_{B_R(x_0)}|u(x)|^p\,dx\geq \(\dfrac{1-\beta}{2}\)^{n+p} \dfrac{(p-1)\,C_2}{(n+1)\,C_1\,\omega_n}$},	
\end{cases}
\end{equation}}
for each $R>0$, $x_0\in \R^n$, and $0<\beta <1$ with positive constants $C_1,C_2$ depending only on $p,n$ and $h$, with $\omega_n=|B_1|$; and with $L_1$ depending only on $p,n$ and $h$, and $L_2$ depending only on $p,n,h$ and $\beta$.

\end{theorem}

\begin{proof}

Our goal is to estimate the supremum of $|u|$ over a ball by the $L^p$-norm of $u$ over a slightly larger ball. To do this, the idea is to use \eqref{eq:third Green identity} and estimate the integrals by integrating \eqref{eq:inequality between G} in $x$.%Let us focus on the case when $h(x)=|x|^p$.

In fact, let us set $\omega=\dfrac{u(y)}{|u(y)|}$ and $r=\delta \,|u(y)|$, with $\delta>0$ to be chosen; $u(y)\neq 0$.
Applying the identity \eqref{eq:third Green identity} with $v(x)\leadsto -G(x-y,u(y))$ and the ball $B_r(y)\leadsto B_r(y+r\,\omega)$
yields
\begin{align}\label{eq:main formula A+B}
&v(y+r\,\omega)=-G(r\,\omega ,u(y))\notag\\
&=-\fint_{B_r(y+r\,\omega)}G\(x-y,u(y)\)\,dx\notag \\
&\qquad +\dfrac{n}{r^n}\,\int_0^r \rho^{n-1}\int_{|x-y-r \omega|\leq \rho}\(\Gamma(x-y-r\omega)-\Gamma(\rho)\)\,\Delta_x\(-G\(x-y,u(y)\)\) \,dx\,d\rho\notag\\
&=A+B.
\end{align}

We first estimate the left hand side of \eqref{eq:main formula A+B}
from below.
Write
\begin{align*}
&-G(r\,\omega ,u(y))\\
&=-G(\delta u(y), u(y))=h(\delta\,u(y)) +h(u(y))-h\(\delta\,u(y)-u(y)\)\\
&=\delta\,\(\dfrac{h(\delta\,u(y))}{\delta} +\dfrac{h(u(y))-h\(\delta\,u(y)-u(y)\)}{\delta}\)\\
&=\delta\,\(\dfrac{h(\delta\,u(y))}{\delta} +\dfrac{h(-u(y))-h\(\delta\,u(y)-u(y)\)}{\delta}\)\quad \text{since $h$ is even}\\ 
&=\delta\,\(\dfrac{h(\delta\,u(y))}{\delta} +\dfrac{\nabla h\(\xi\)\cdot -\delta\,u(y)}{\delta}\),\quad \text{with $\xi$ an intermediate point between $-u(y)$ and $\delta u(y)-u(y)$}.
\end{align*}
Since $h$ is smooth and homogenous of degree $p>1$, i.e., $h(\lambda x)=\lambda^p\,h(x)$ for $\lambda>0$, it follows that $\nabla h(\lambda x)=\lambda^{p-1}\,\nabla h(x)$ and so
\begin{align*}
\dfrac{h(\delta\,u(y))}{\delta} +\dfrac{\nabla h\(\xi\)\cdot -\delta\,u(y)}{\delta}
&=
\dfrac{h\(\delta |u(y)|\,\dfrac{u(y)}{|u(y)|}\)}{\delta}
-\nabla h\(\xi\)\cdot u(y)\\
&=
\delta^{p-1}\,|u(y)|^p\,h\(\dfrac{u(y)}{|u(y)|}\)-\nabla h\(|\xi|\dfrac{\xi}{|\xi|}\)\cdot u(y)\\
&=
\delta^{p-1}\,|u(y)|^p\,h\(\dfrac{u(y)}{|u(y)|}\)-|\xi|^{p-1}\nabla h\(\dfrac{\xi}{|\xi|}\)\cdot u(y)\\
&=
\delta^{p-1}\,|u(y)|^p\,h\(\dfrac{u(y)}{|u(y)|}\)-|u(y)|^p\,\(\dfrac{|\xi|}{|u(y)|}\)^{p-1}\nabla h\(\dfrac{\xi}{|\xi|}\)\cdot \dfrac{u(y)}{|u(y)|}\\
&=
|u(y)|^p\,\(\delta^{p-1}\,h\(\dfrac{u(y)}{|u(y)|}\)-\(\dfrac{|\xi|}{|u(y)|}\)^{p-1}\nabla h\(\dfrac{\xi}{|\xi|}\)\cdot \dfrac{u(y)}{|u(y)|}\):=|u(y)|^p\,f(\delta,y).
\end{align*}
If $\delta\to 0^+$ we get $\xi\to -u(y)$ and
\[
f(\delta,y)=\delta^{p-1}\,h\(\dfrac{u(y)}{|u(y)|}\)-\(\dfrac{|\xi|}{|u(y)|}\)^{p-1}\nabla h\(\dfrac{\xi}{|\xi|}\)\cdot \dfrac{u(y)}{|u(y)|}\to 
-\nabla h\(\dfrac{-u(y)}{|u(y)|}\)\cdot \dfrac{u(y)}{|u(y)|}.
\]
Since $h$ is convex, then for each $x_0$ and $x$ we have $h(x)\geq h(x_0)+\nabla h(x_0)\cdot (x-x_0)$. Applying this inequality with $x_0=\dfrac{-u(y)}{|u(y)|}$ and $x=0$ yields
\[
h(0)\geq h\(\dfrac{-u(y)}{|u(y)|}\)+\nabla h\(\dfrac{-u(y)}{|u(y)|}\)\cdot \dfrac{u(y)}{|u(y)|}
\]
and since $h(0)=0$,
\[
h\(\dfrac{-u(y)}{|u(y)|}\)\leq -\nabla h\(\dfrac{-u(y)}{|u(y)|}\)\cdot \dfrac{u(y)}{|u(y)|}.
\]
{\it If $h$ is strictly positive in the unit sphere, then 
$$0<m=\min_{x\in S^{n-1}} h(x)\leq M=\max_{x\in S^{n-1}} h(x)$$ 
by continuity}.
Therefore we get the inequality 
\[
0<m\leq -\nabla h\(\dfrac{-u(y)}{|u(y)|}\)\cdot \dfrac{u(y)}{|u(y)|}\leq \max_{x\in S^{n-1}}|\nabla h(x)|.
\]
We next show that $f(\delta,y)\to -\nabla h\(\dfrac{-u(y)}{|u(y)|}\)\cdot \dfrac{u(y)}{|u(y)|}$ as $\delta\to 0^+$ uniformly in $y\neq 0$.
In fact,
\begin{align*}
f(\delta,y)+\nabla h\(\dfrac{-u(y)}{|u(y)|}\)\cdot \dfrac{u(y)}{|u(y)|}
&=\delta^{p-1}\,h\(\dfrac{u(y)}{|u(y)|}\)\\
&\qquad  -\(\dfrac{|\xi|}{|u(y)|}\)^{p-1}\nabla h\(\dfrac{\xi}{|\xi|}\)\cdot \dfrac{u(y)}{|u(y)|}
+
\nabla h\(\dfrac{-u(y)}{|u(y)|}\)\cdot \dfrac{u(y)}{|u(y)|}=D_1+D_2.
\end{align*}
We have $D_1\leq M\,\delta^{p-1}$, and from the homogeneity of $\nabla h$
\begin{align*}
D_2
&=
-\nabla h\(\dfrac{\xi}{|u(y)|}\)\cdot \dfrac{u(y)}{|u(y)|}+\nabla h\(\dfrac{-u(y)}{|u(y)|}\)\cdot \dfrac{u(y)}{|u(y)|},
\end{align*}
so
\[
|D_2|\leq \left|\nabla h\(\dfrac{\xi}{|u(y)|}\)-\nabla h\(\dfrac{-u(y)}{|u(y)|}\)\right|.
\]
Since $\xi$ is an intermediate point between $-u(y)$ and $\delta u(y)-u(y)$,
$\xi=-u(y)+t\,\delta \,u(y)$ for some $0<t<1$, so 
$\left| \dfrac{\xi}{|u(y)|}-\dfrac{-u(y)}{|u(y)|}\right|<\delta$. Since $\nabla h$ is uniformly continuous in a neighborhood of $S^{n-1}$ the uniform convergence of $f$ follows.

Therefore, we get the following lower bound for the left hand side of \eqref{eq:main formula A+B}:
there exists $\delta_0>0$ depending only on $h$ and independent of $y$ such that 
\begin{equation}\label{eq:lower estimate general cost}
-G(r\,\omega ,u(y))\geq \dfrac{m}{2}\,\delta\,|u(y)|^p, \qquad \text{for $0<\delta<\delta_0$,}
\end{equation}
with $\omega=u(y)/|u(y)|$ and $r=\delta\, |u(y)|$, for each $y$ with $u(y)\neq 0$.
On the other hand, if $\delta\geq \delta_0$, then $\dfrac{r}{|u(y)|}\geq \delta_0$,  implying obviously that $|u(y)|\leq \dfrac{r}{\delta_0}$, and obtaining the bound 
$|u(y)|\leq \dfrac{\alpha}{\delta_0}$ for $0<r\leq \alpha$.

We now turn to estimate the right hand side of \eqref{eq:main formula A+B}.
Let us first calculate $\Delta_z G(z,v)$:
\[
\Delta_z G(z,v)=\Delta h(z-v)-\Delta h(z).
\]
Hence
\[
\Delta_x\(-G\(x-y,u(y)\)\)
=
-(\Delta_zG)(x-y,u(y))
=
\Delta h(x-y)-\Delta h(x-y-u(y)),
\]
and so 
\[
B=
\dfrac{n}{r^n}\,\int_0^r \rho^{n-1}
\int_{|x-y-r \omega|\leq \rho}\(\Gamma(x-y-r\omega)-\Gamma(\rho)\)\,
\(\Delta h(x-y)-\Delta h(x-y-u(y))\) \,dx\,d\rho.
\]
Let us analyze the inner integral
\[
I(\rho,r,y)=
\int_{|x-y-r \omega|\leq \rho}\(\Gamma(x-y-r\omega)-\Gamma(\rho)\)\,
\(\Delta h(x-y)-\Delta h(x-y-u(y))\) \,dx.
\]
Making the change of variables $z=x-y-r\,\omega$ yields
\begin{align*}
I(\rho,r,y)
&=
\int_{|z|\leq \rho}\(\Gamma(z)-\Gamma(\rho)\)\,
\(\Delta h(z+r\,\omega)-\Delta h\(z+r\,\omega-u(y)\)\) \,dz.
\end{align*}
We have that $\Delta h$ is homogenous of degree $p-2$ so 
\[
\Delta h(z+r\,\omega)=
\Delta h\(|z+r\,\omega|\,\dfrac{z+r\,\omega}{|z+r\,\omega|}\)
=
|z+r\,\omega|^{p-2}\,\Delta h\(\dfrac{z+r\,\omega}{|z+r\,\omega|}\).
\]
Write, with $e_1$ a fixed unit vector in $S^{n-1}$, 
\begin{align*}
&\int_{|z|\leq \rho}\(\Gamma(z)-\Gamma(\rho)\)\,
\Delta h(z+r\,\omega)\,dz\\
&=\int_{|z|\leq \rho}\(\Gamma(z)-\Gamma(\rho)\)\,
|z+r\,\omega|^{p-2}\,\Delta h\(\dfrac{z+r\,\omega}{|z+r\,\omega|}\) \,dz\\
&=
\int_{|v|\leq \rho}\(\Gamma(v)-\Gamma(\rho)\)\,
|Tv+r\,Te_1|^{p-2}\,\Delta h\(\dfrac{Tv+r\,Te_1}{|Tv+r\,Te_1|}\) \,dv,\,\text{with $T$ rotation around $0$ with $Te_1=\omega$}\\
&=
\int_{|v|\leq \rho}\(\Gamma(v)-\Gamma(\rho)\)\,
|v+r\,e_1|^{p-2}\,\Delta h\(\dfrac{Tv+r\,Te_1}{|Tv+r\,Te_1|}\) \,dv.
\end{align*}
{\small Similarly,
\begin{align*}
&\int_{|z|\leq \rho}\(\Gamma(z)-\Gamma(\rho)\)\,
\Delta h\(z+r\,\omega-u(y)\) \,dz\\
&=\int_{|z|\leq \rho}\(\Gamma(z)-\Gamma(\rho)\)\,
|z+r\,\omega-u(y)|^{p-2}\,\Delta h\(\dfrac{z+r\,\omega-u(y)}{|z+r\,\omega-u(y)|}\) \,dz\\
&=
\int_{|v|\leq \rho}\(\Gamma(v)-\Gamma(\rho)\)\,
|Tv+r\,Te_1-u(y)|^{p-2}\,\Delta h\(\dfrac{Tv+r\,Te_1-u(y)}{|Tv+r\,Te_1-u(y)|}\) \,dv,\,\text{with $T$ rotation around $0$ with $Te_1=\omega$}\\
&=
\int_{|v|\leq \rho}\(\Gamma(v)-\Gamma(\rho)\)\,
|Tv+r\,Te_1-|u(y)|Te_1|^{p-2}\,\Delta h\(\dfrac{Tv+r\,Te_1-|u(y)|Te_1}{|Tv+r\,Te_1-|u(y)|Te_1|}\) \,dv\\
&=
\int_{|v|\leq \rho}\(\Gamma(v)-\Gamma(\rho)\)\,
|v+(r-|u(y)|)\,e_1|^{p-2}\,\Delta h\(\dfrac{Tv+(r-|u(y)|)Te_1}{|Tv+(r-|u(y)|)Te_1|}\) \,dv,\end{align*}
since $\omega=u(y)/|u(y)|$.
Then
\begin{align*}
I(\rho,r,y)
=
\int_{|v|\leq \rho}\(\Gamma(v)-\Gamma(\rho)\)\,
\(|v+r\,e_1|^{p-2}\,\Delta h\(\dfrac{Tv+r\,Te_1}{|Tv+r\,Te_1|}\)
-
|v+(r-|u(y)|)\,e_1|^{p-2}\,\Delta h\(\dfrac{Tv+(r-|u(y)|)Te_1}{|Tv+(r-|u(y)|)Te_1|}\)
\) \,dv.\end{align*}
We then get 
\begin{align*}
B&=
\dfrac{n}{r^n}\,\int_0^r \rho^{n-1}\,
I(\rho,r,y)\,d\rho=
n\,
\int_0^1t^{n-1}\,I(r\,t,r,y)\,dt.
\end{align*}
}
Now making the change of variables $v=r\zeta$ in the integral $I$ yields
{\small \begin{align*}
&I(r\,t,r,y)\\
&=
\int_{|\zeta|\leq t}
\(\Gamma(r\,\zeta)-\Gamma(r\,t)\)\,
\(|r\,\zeta+r\,e_1|^{p-2}\,\Delta h\(\dfrac{T(r\,\zeta)+r\,Te_1}{|T(r\,\zeta)+r\,Te_1|}\)
-
|r\,\zeta+(r-|u(y)|)\,e_1|^{p-2}\,\Delta h\(\dfrac{T(r\,\zeta)+(r-|u(y)|)Te_1}{|T(r\,\zeta)+(r-|u(y)|)Te_1|}\)
\) \,r^n\,d\zeta\\
&=r^p\,\int_{|\zeta|\leq t}
\(\Gamma(\zeta)-\Gamma(t)\)\,
\(|\zeta+e_1|^{p-2}\,\Delta h\(\dfrac{T(\zeta)+Te_1}{|T(\zeta)+Te_1|}\)
-
|\zeta+(1-|u(y)|/r)\,e_1|^{p-2}\,\Delta h\(\dfrac{T(\zeta)+(1-|u(y)|/r)Te_1}{|T(\zeta)+(1-|u(y)|/r)Te_1|}\)
\) \,d\zeta
\end{align*}
}
and now letting $r=\delta\,|u(y)|$ as before yields 
\begin{align*}
B
&=
n\,
|u(y)|^p\,\delta^p\,\int_0^1t^{n-1}
\int_{|\zeta|\leq t}
\(\Gamma(\zeta)-\Gamma(t)\)\,\\
&\qquad 
\(|\zeta+e_1|^{p-2}\,\Delta h\(\dfrac{T(\zeta)+Te_1}{|T(\zeta)+Te_1|}\)
-
|\zeta+(1-(1/\delta))\,e_1|^{p-2}\,\Delta h\(\dfrac{T(\zeta)+(1-(1/\delta))Te_1}{|T(\zeta)+
(1-(1/\delta))Te_1|}\)
\) \,d\zeta\,dt\\
&=
n\,
|u(y)|^p\,\delta^p\,\int_0^1t^{n-1}
\int_{|\zeta|\leq t}
\(\Gamma(\zeta)-\Gamma(t)\)\,
 |\zeta+e_1|^{p-2}\,\Delta h\(\dfrac{T(\zeta)+Te_1}{|T(\zeta)+Te_1|}\)\,d\zeta\,dt\\
 &\qquad 
 -n\,
|u(y)|^p\,\delta^2\,\int_0^1t^{n-1}
\int_{|\zeta|\leq t}
\(\Gamma(\zeta)-\Gamma(t)\)\,
|\delta\,\zeta+(\delta-1)\,e_1|^{p-2}\,\Delta h\(\dfrac{\delta\,T(\zeta)+(\delta-1)Te_1}{|\delta\,T(\zeta)+
(\delta-1)Te_1|} \)\,d\zeta\,dt \\
&=
n\,
|u(y)|^p\,\delta\,F(\delta),
\end{align*}
where 
\begin{align*}
F(\delta)
&=
\delta^{p-1}\,\int_0^1t^{n-1}
\int_{|\zeta|\leq t}
\(\Gamma(\zeta)-\Gamma(t)\)\,
 |\zeta+e_1|^{p-2}\,\Delta h\(\dfrac{T(\zeta)+Te_1}{|T(\zeta)+Te_1|}\)\,d\zeta\,dt\\
 &\qquad 
 -\delta\,\int_0^1t^{n-1}
\int_{|\zeta|\leq t}
\(\Gamma(\zeta)-\Gamma(t)\)\,
|\delta\,\zeta+(\delta-1)\,e_1|^{p-2}\,\Delta h\(\dfrac{\delta\,T(\zeta)+(\delta-1)Te_1}{|\delta\,T(\zeta)+
(\delta-1)Te_1|} \)\,d\zeta\,dt.
\end{align*}
Since $\Delta h$ is continuous, it is bounded in $S^{n-1}$ and so $F(\delta)\to 0$ uniformly in $y$ as $\delta\to 0^+$ when $p\geq 2$.
Therefore there exists $\delta_1>0$ such that $F(\delta)\leq \dfrac{m}{4\,n}$ for $0<\delta\leq \delta_1$ and so 
\[
B\leq \dfrac{m}{4}\,|u(y)|^p\,\delta
\]
for $0<\delta\leq \delta_1$.
Combining this with \eqref{eq:lower estimate general cost} and \eqref{eq:main formula A+B}
yields the inequality 
\begin{equation}\label{eq:lower bound of A}
\dfrac{m}{4}\,|u(y)|^p\,\delta\leq A,\qquad \text{for $0<\delta<\bar \delta$} 
\end{equation}
with $\bar \delta=\min\{\delta_0,\delta_1\}$ independent of $y$ -depending only on $n,p$ and $h$- and with $r=\delta\,|u(y)|$.

We next estimate $A$ from above. To do this will use \eqref{eq:inequality between G}.
From \eqref{eq:main formula A+B} 
\begin{align*}
A
&=
-\fint_{B_r(y+r\,\omega)}G\(x-y,u(y)\)\,dx
\leq \fint_{B_r(y+r\,\omega)}\(G\(y-x,u(x)\)+2\,h(x-y)\)\,dx\\
&=
\fint_{B_r(y+r\,\omega)}G\(y-x,u(x)\)\,dx
+
2\,\fint_{B_r(y+r\,\omega)}h(x-y)\,dx\\
&=
\fint_{B_r(y+r\,\omega)} \(h(y-x-u(x))-h(y-x)-h(u(x))\)\,dx
+
2\,\fint_{B_r(y+r\,\omega)}h(x-y)\,dx\\
&=
\fint_{B_r(y+r\,\omega)} h(y-x-u(x))\,dx
-
\fint_{B_r(y+r\,\omega)}h(u(x))\,dx
+
\fint_{B_r(y+r\,\omega)}h(x-y)\,dx,\quad \text{since $h$ is even}\\
&\leq
\fint_{B_r(y+r\,\omega)} h(y-x-u(x))\,dx
+
\fint_{B_r(y+r\,\omega)}h(x-y)\,dx,\quad \text{since $h\geq 0$}\\
&=A_1+A_2.
\end{align*}

Let us estimate $A_i$:
\begin{align*}
A_1
&
=
\fint_{B_r(y+r\,\omega)}h\(|y-x-u(x)|\,\dfrac{y-x-u(x)}{|y-x-u(x)|}\)\,dx\\
&=
\fint_{B_r(y+r\,\omega)}|y-x-u(x)|^p\,h\(\dfrac{y-x-u(x)}{|y-x-u(x)|}\)\,dx\\
&\leq
\max_{x\in S^{n-1}}h(x)\,\fint_{B_r(y+r\,\omega)}|y-x-u(x)|^p\,dx\\
&\leq
M\,\fint_{B_r(y+r\,\omega)}2^{p-1}\(|y-x|^p+|u(x)|^p\)\,dx\\
&=
2^{p-1}\,M\,\fint_{B_r(y+r\,\omega)}|y-x|^p\,dx
+
2^{p-1}\,M\,\fint_{B_r(y+r\,\omega)}|u(x)|^p\,dx;
\end{align*}
%\begin{align*}
%A_2
%&=
%\fint_{B_r(y+r\,\omega)}h\(|u(x)|\,\dfrac{u(x)}{|u(x)|}\)\,dx
%=
%\fint_{B_r(y+r\,\omega)}|u(x)|^p\,h\(\dfrac{u(x)}{|u(x)|}\)\,dx
%\leq
%M\,\fint_{B_r(y+r\,\omega)}|u(x)|^p\,dx;
%\end{align*}
\begin{align*}
A_2
&=
\fint_{B_r(y+r\,\omega)}h\(|x-y|\,\dfrac{x-y}{|x-y|}\)\,dx
=
\fint_{B_r(y+r\,\omega)}|x-y|^p\,h\(\dfrac{x-y}{|x-y|}\)\,dx
\leq
M\,\fint_{B_r(y+r\,\omega)}|x-y|^p\,dx.
\end{align*}
We then obtain
\[
A\leq
2^{p-1}\,M\,\fint_{B_r(y+r\,\omega)}|u(x)|^p\,dx+\(2^{p-1}+1\)M\,\fint_{B_r(y+r\,\omega)}|x-y|^p\,dx,
\]
with $M=\max_{x\in S^{n-1}} h(x)$.
We have
\begin{align*}
\fint_{B_r(y+r\,\omega)}|x-y|^p\,dx
&=
\dfrac{1}{|B_r(0)|}\int_{|x-y-r\omega|\leq r}|x-y|^p\,dx\\
&=
\dfrac{1}{|B_r(0)|}\int_{|z|\leq 1}|r(z+\omega)|^p\,r^n\,dz\qquad \text{with $rz=x-y-r\omega$}\\
&=
r^p\,\fint_{B_1(0)}|z+\omega|^p\,dz\leq 2^p\,r^p.
\end{align*}
Let us now fix a ball $B_R(x_0)$, and
suppose $y\in B_{\beta\,R}(x_0)$ with $0<\beta<1$, $R>0$. Then 
$B_r(y+r\,\omega) \subset B_R(x_0)$ for $r\leq \dfrac{1-\beta}{2}R$ and so
\[
\fint_{B_r(y+r\,\omega)}|u(x)|^p\,dx
\leq
\dfrac{1}{|B_r(0)|}\,\int_{B_R(x_0)}|u(x)|^p\,dx.
\]
Combining these estimates with the lower bound \eqref{eq:lower bound of A} and the upper bound for $A$ 
we obtain
\[
\dfrac{m}{4}\,|u(y)|^p\,\delta\leq \dfrac{M_1}{r^n}\,\int_{B_R(x_0)}|u(x)|^p\,dx+M_2\,r^p,\qquad \text{for $0<\delta<\bar \delta$} 
\]
with $\bar \delta$ structural constant independent of $y$ and with $r=\delta\,|u(y)|$, for $y\in B_{\beta\,R}(x_0)$ and $r\leq (1-\beta)R/2$; $M_1=2^{p-1}M/\omega_n$, $M_2=2^p(2^{p-1}+1)M$.
Therefore, if $y\in B_{\beta\,R}(x_0)$, $0<r\leq (1-\beta)R/2$, and $\delta=\dfrac{r}{|u(y)|}<\bar \delta$, then we obtain the bound 
\[
|u(y)|^{p-1}
\leq
\dfrac{C_1}{r^{n+1}}\,\int_{B_{R}(x_0)}|u(x)|^p\,dx+C_2\,r^{p-1}:=H(r),
\]
with $C_i$ constants depending only on $p, n$, and $M/m$; $C_1=\dfrac{2^{p+1}}{\omega_n}(M/m)$, $C_2=2^{p+2}(2^{p-1}+1)(M/m)$.
On the other hand, if $y\in B_{\beta\,R}(x_0)$, $0<r\leq (1-\beta)R/2$,  and $\delta=\dfrac{r}{|u(y)|}\geq \bar \delta$, then 
\[
|u(y)|\leq \dfrac{r}{\bar \delta}\leq \dfrac{1-\beta}{2\,\bar \delta}R.
\] 
So for any $y\in B_{\beta\,R}(x_0)$ and any $0<r\leq (1-\beta)R/2$ we obtain
\begin{align*}
|u(y)|
\leq 
\max \left\{H(r)^{1/(p-1)}, \dfrac{r}{\bar \delta} \right\}.
\end{align*}
Since the constant $C_2$ in the definition of $H(r)$ can be enlarged with the last estimate remaining to hold, we can take $C_2$ so that $C_2\geq 1/\bar \delta^{p-1}$ and in this way $H(r)^{1/(p-1)}\geq \dfrac{r}{\bar \delta}$, and so $\max \left\{H(r)^{1/(p-1)}, \dfrac{r}{\bar \delta} \right\}=H(r)^{1/(p-1)}$.
Therefore we obtain the estimate 
\begin{equation}\label{eq:main estimate in r to iterate}
\sup_{y\in B_{\beta\,R}(x_0)}|u(y)|
\leq 
\min_{0<r\leq (1-\beta)R/2} H(r)^{1/(p-1)}.
\end{equation}
Set 
\[
\Delta=\int_{B_R(x_0)}|u(x)|^p\,dx,
\]
so $H(r)= C_1\,\Delta\,r^{-(n+1)}+C_2\,r^{p-1} $. The minimum of $H$ over $(0,\infty)$ is attained at 
\[
r_0=\(\dfrac{(n+1)\,C_1\,\Delta}{(p-1)\,C_2}\)^{1/(n+p)},
\]
$H$ is decreasing in $(0,r_0)$ and increasing in $(r_0,\infty)$, and
\[
\min_{[0,\infty)}H(r)=H(r_0)=
\(\(\dfrac{n+1}{p-1}\)^{-(n+1)/(n+p)}+
\(\dfrac{n+1}{p-1}\)^{(p-1)/(n+p)}\)\,\(C_1\,\Delta\)^{(p-1)/(n+p)}\,C_2^{(n+1)/(n+p)}.
\]
If $r_0<(1-\beta)R/2$, then $\min_{0<r<(1-\beta)R/2} H(r)=H(r_0)$. 
On the other hand, if $r_0>(1-\beta)R/2$, that is, 
$\Delta\geq \(\dfrac{1-\beta}{2}R\)^{n+p} \dfrac{(p-1)\,C_2}{(n+1)\,C_1} :=\Delta_0$, then we have
\begin{align*}
\min_{0<r<(1-\beta)R/2} H(r)&=H\(\dfrac{1-\beta}{2}R\)
=
C_1\,\Delta\,\(\dfrac{1-\beta}{2}R\)^{-(n+1)}+
C_2\,\(\dfrac{1-\beta}{2}R\)^{p-1}\\
&= 
C_1\,\Delta\,\(\dfrac{1-\beta}{2}R\)^{-(n+1)}+
C_2\,\Delta\,\dfrac{1}{\Delta}\,\(\dfrac{1-\beta}{2}R\)^{p-1}\\
&\leq
C_1\,\Delta\,\(\dfrac{1-\beta}{2}R\)^{-(n+1)}
+
\dfrac{n+1}{p-1}\,C_1\,\Delta\,\(\dfrac{1-\beta}{2}R\)^{-(n+1)}\\
&=
C_1\,\dfrac{p+n}{p-1}\,\(\dfrac{1-\beta}{2}R\)^{-(n+1)}\Delta:=K_2\,R^{-(n+1)}\,\Delta.
\end{align*}
We then obtain the following estimate valid for all $0<\beta<1$
\begin{equation}\label{eq:main L infty estimate general cost}
\sup_{y\in B_{\beta\,R}(x_0)}|u(y)|^{p-1}
\leq
\begin{cases}
K_1\,\Delta^{(p-1)/(n+p)} & \text{if $\Delta\leq \Delta_0$}\\
K_2\,R^{-(n+1)}\,\Delta  & \text{if $\Delta\geq \Delta_0$,}	
\end{cases}
\end{equation}
with $K_1= \(\(\dfrac{n+1}{p-1}\)^{-(n+1)/(n+p)}+
\(\dfrac{n+1}{p-1}\)^{(p-1)/(n+p)}\)\,C_1^{(p-1)/(n+p)}\,C_2^{(n+1)/(n+p)}$, 
$K_2=C_1\,\dfrac{p+n}{p-1}\,\(\dfrac{1-\beta}{2}\)^{-(n+1)}$, and
$\Delta= \int_{B_R(x_0)}|u(x)|^p\,dx$.

This completes the proof of the theorem.
\end{proof}

\begin{remark}\rm
Suppose $x_0\in \R^n$, $\displaystyle \lim_{R\to 0^+}\dfrac{1}{R^p}\fint_{B_R(x_0)}|u(x)|^p\,dx=0$ and $x_0$ is a Lebesgue point of $|u(x)|^p$.
Then \eqref{eq:main L infinity estimate in a general ball} implies that $u(x)$ is Lipschitz at $x_0$.
In fact, first notice that since $x_0$ is a Lebesgue point, the condition on the limit implies $u(x_0)=0$. 
Now, pick for example $\beta=1/2$. Then there exists $R_0>0$ such that
\[
\dfrac{1}{R^p}\fint_{B_R(x_0)}|u(x)|^p\,dx
\leq \(\dfrac{1}{4}\)^{n+p}\dfrac{(p-1)\,C_2}{(n+1)\,C_1\,\omega_n},
\quad \text{for $0<R<R_0$}
\]
and so 
$\sup_{B_{R/2}(x_0)}|u(x)|\leq C_0\,R$ from \eqref{eq:main L infinity estimate in a general ball}
for $0<R<R_0$, with $C_0$ a positive constant depending only on $n,p$ and $h$.
If $y\in B_{R_0/2}(x_0)$ and $R=2\,|y-x_0|$, then $|u(y)|\leq \sup_{B_{|y-x_0|}(x_0)}|u(x)|\leq 2\,C_0\,|y-x_0|$. In particular, this implies $|Ty-Tx_0|\leq C\,|y-x_0|$ for $y\in B_{R_0/2}(x_0)$. 
\end{remark}

\section{Estimates for the displacement interpolating map}\label{sec:Estimates for the displacement interpolating map}
\setcounter{equation}{0}
In order to prove the desired estimates we first give a condition equivalent to \eqref{eq:map T is c-monotone h} resembling the classical notion of monotone map.
In fact, from \eqref{eq:map T is c-monotone h} we can write
\begin{align*}
0&\leq  h(y-Tx)-h(y-Ty)-\(h(x-Tx)-h(x-Ty)\)\notag\\
&=\int_0^1 \langle Dh(y-Ty+s(Ty-Tx)), Ty-Tx\rangle ds -\int_0^1 \langle Dh(x-Ty+s(Ty-Tx)), Ty-Tx\rangle ds\notag\\
&=\int_0^1 \langle Dh(y-Ty+s(Ty-Tx))-Dh((x-Ty+s(Ty-Tx)), Ty-Tx\rangle ds\notag\\
&=-  \int_0^1\int_0^1\langle D^2h(x-Ty+s(Ty-Tx)+t (x-y)) (y-x),  (Ty-Tx)\rangle dt\, ds\notag\\
&=- \int_0^1\int_0^1\langle D^2h(y-Ty+s(Ty-Tx)+t (x-y)) (x-y),  (Ty-Tx)\rangle dt\, ds\notag\\
&=  \langle A(x,y) (x-y),  Tx-Ty\rangle.
\end{align*}
Therefore \eqref{eq:map T is c-monotone h} is equivalent to 
\begin{equation}\label{eq:A}
\langle A(x,y) (x-y),  Tx-Ty\rangle\geq 0
\end{equation}
with
\begin{equation}\label{eq:definition of A(x,y)}
A(x,y)=\int_0^1\int_0^1 D^2h(y-Ty+s(Ty-Tx)+t (x-y))dt\, ds.
\end{equation}
Let us analyze the matrix $A(x,y)$. $A(x,y)$ is clearly symmetric, and satisfies $A(x,y)=A(y,x)$ by changing variables in the integral.
If $h$ is homogenous of degree $p$ with $p\geq 2$, then $D^2h(z)$ is homogeneous of degree $p-2$, i.e., $D^2h(\mu\,z)=\mu^{p-2}D^2h(z)$ for all $\mu>0$. 
In addition, if $h$ {\it is strictly convex}, then $D^2h(x)$ is positive definite for each $x\in S^{n-1}$, i.e, there is a constant $\lambda>0$ such that 
\[
\left\langle D^2h(x)\,\xi,\xi \right\rangle\geq \lambda\,|\xi|^2
\]
for all $x\in S^{n-1}$ and all $\xi\in \R^n$. Since $h$ is $C^2$, then there is also 
a positive constant $\Lambda$ such that 
\begin{equation}\label{eq:strict ellipticity of D2h}
\lambda \,|\xi|^2
\leq 
\left\langle D^2h(x)\xi,\xi\right\rangle
\leq
\Lambda \,|\xi|^2,\qquad \forall x\in S^{n-1},\xi\in \R^n.
\end{equation}
We then have
\[
A(x,y)=
\int_0^1 \int_0^1|y-Ty+s(Ty-Tx)+t (x-y)|^{p-2}D^2h\left(\frac{y-Ty+s(Ty-Tx)+t (x-y)}{ |y-Ty+s(Ty-Tx)+t (x-y)|}\right)dt \, ds
\]
and
\begin{equation}\label{eq:ellipticity of A}
\lambda \,\Phi(x,y)\,|\xi|^2
\leq
\left\langle  A(x,y)\,\xi,\xi\right\rangle\leq \Lambda \, \Phi(x,y)\,|\xi|^2\quad \forall \xi \in \R^n,
\end{equation}
with 
\begin{equation}\label{eq:Phi}
 \Phi(x,y)=\int_0^1\int_0^1|y-Ty+s(Ty-Tx)+t (x-y)|^{p-2}dt \, ds.
\end{equation}
We also have that $\Phi(x,y)=0$ if and only if $y-Ty+s(Ty-Tx)+t (x-y)=0$ for all $s,t\in [0,1]$. That is, $\Phi(x,y)=0$ if and only if $y-Ty=0$, $Ty-Tx=0$ and $x-y=0$.
Therefore $\Phi(x,y)>0$ if and only if $Ty\neq y$ or $Ty\neq Tx$ or $x\neq y$.
\begin{remark}\rm\label{rmk:for pneq 2 C_4 does not hold}
If $c(x,y)=|x-y|^p$, then $\nabla_{xy}c(x,y)=-p\,|x-y|^{p-2}\(Id+(p-2)\,\(\dfrac{x-y}{|x-y|}\otimes \dfrac{x-y}{|x-y|}\)\)$ and from the Sherman-Morrison formula
it follows that $\det \nabla_{xy}c(x,y)=(p-1)\(-p\,|x-y|^{p-2}\)^n$. So condition \cite[($C_4$)]{2020-otto-prodhomme-ried-general-costs} does not hold for $p\neq 2$.

\end{remark}
\begin{remark}\rm
To illustrate the notion of $h$-monotonicity, suppose $T$ satisfies \eqref{eq:A} and is $C^1$. Then writing $y=x+\delta\,\o$ with $|\o|=1$ yields
\[
A(x,x+\delta\,\o)=
\iint_{[0,1]^2} D^2h\(x+\delta\,\o-T(x+\delta\,\o)+s(T(x+\delta\,\o)-Tx)+t (-\delta\,\o)\)dt\, ds
\to D^2h\(x-Tx\)
\]
as $\delta\to 0$ and 
\[
\left\langle A(x,x+\delta\,\o) (-\delta\,\o),  Tx-T(x+\delta\,\o)\right\rangle\geq 0.
\]
Dividing the last expression by $\delta^2$ and letting $\delta\to 0$ we obtain
\[
\left\langle D^2h\(x-Tx\) \o,  \dfrac{\partial T}{\partial x}(x)\o\right\rangle\geq 0,
\]
where $\dfrac{\partial T}{\partial x}$ is the Jacobian matrix of $T$ evaluated at $x$.
Since $h$ is $C^2$, the matrix $D^2h$ is symmetric and we get
\[
\left\langle  \o,  D^2h\(x-Tx\)\,\dfrac{\partial T}{\partial x}(x)\o\right\rangle\geq 0
\]
for each unit vector $\o$. 
Therefore, if $T$ is $h$-monotone and $C^1$, the matrix $D^2h\(x-Tx\)\,\dfrac{\partial T}{\partial x}(x)$ is positive semidefinite for each $x$;
notice that $\dfrac{\partial T}{\partial x}(x)$ is not necessarily symmetric.
In particular, when $n=1$, $T$ is $h$-monotone if and only if $T$ is non decreasing.

\end{remark}

\subsection{$L^\infty$-estimates of the interpolating map}\label{sec:L-infty-estimates of the interpolating map}
Let $T$ be a $h$-monotone map, i.e., satisfies \eqref{eq:map T is c-monotone h}, and consider the interpolating map defined by 
\begin{equation}\label{eq:definition of interpolanting map}
T_tx=t\,Tx+(1-t)\,x,\quad 0\leq t\leq 1.
\end{equation}

\begin{theorem}
Suppose the assumptions of Theorem \ref{thm:main Linfty estimate} hold and assume in addition that $h$ is strictly convex.
If the integral $\mathcal E=\int_{B_1(0)}|Tx-x|^p\,dx$ is sufficiently small, then given $0<\beta<1$ there exists $0<\beta <\bar \beta<1$ depending only on $\beta$ and the ellipticity constants $\lambda,\Lambda$ in \eqref{eq:strict ellipticity of D2h} such that 
\begin{equation}\label{eq:inclusion T_t^-1 ball contained into a ball}
T_t^{-1}\(B_\beta (0)\)\subset B_{\bar \beta}(0)\quad \text{for all $0\leq t\leq 1$},
\end{equation}
that is, $\bigcup_{0\leq t\leq 1}T_t^{-1}\(B_\beta (0)\)\subset B_{\bar \beta}(0)$.
\end{theorem}

\begin{proof}
The inclusion is obvious if $t=0$.
Let $x\in T_t^{-1}\(B_\beta (0)\)$. If $|x|\leq \beta$, then we are done.
Let $\beta<\beta_0<1$, consider the ball $B_{\beta_0}(0)$, and suppose that  $|x|>\beta_0$. 
From \eqref{eq:main L infinity estimate in a general ball} applied in $B_1(0)$, we will show that  is not possible if $\mathcal E$ is sufficiently small, i.e., smaller than $\dfrac{\lambda}{2\,\Lambda}(\beta_0-\beta)$.
We have $y=T_tx\in B_\beta (0)$, and $B_r(y)\subset B_{\beta_0}(0)$ with $r=\beta_0-\beta$.
Let $[y,x]$ be the straight segment between $y$ and $x$, and 
let $z\in \partial B_r(y)\cap [y,x]$. 
So $|z-y|=r$, and $|z|<\beta_0$.
Applying \eqref{eq:A} at $x,z$ yields
\begin{align*}
0&\leq
\left\langle A(x,z) (Tz-Tx), z-x\right\rangle=
\left\langle A(x,z) (Tz-z), z-x\right\rangle
+
\left\langle A(x,z) \(z-Tx\), z-x\right\rangle
\\
&=
\left\langle A(x,z) (Tz-z), z-x\right\rangle
+
\left\langle A(x,z) \(\dfrac1t\,(z-y)+\(1-\dfrac1t\)\,(z-x)\), z-x\right\rangle
\quad \text{since $Tx=\dfrac1t\,y+\(1-\dfrac1t\)\,x$}\\
&=
\left\langle A(x,z) (Tz-z), z-x\right\rangle
+
\dfrac1t\,\left\langle A(x,z) \(z-y\),z-x\right\rangle
+
\(1-\dfrac1t\)\,\left\langle A(x,z) \(z-x\), z-x\right\rangle
\\
&=
\Delta.
\end{align*}
%Now notice that $z-Tx\neq 0$. Otherwise, $z-Tx= 0$ and since $Tx=\dfrac1t T_tx-
%\(\dfrac1t -1\)\,x$, we would have $z=\dfrac1t y-
%\(\dfrac1t -1\)\,x$ implying $y=t\,z+(1-t)\,x$, that is, $y$ would be on the segment $[z,x]$ which is impossible.
Since $x\neq z$, it follows from \eqref{eq:Phi} that $\Phi(x,z)> 0$.
Also notice that $\left\langle A (z-x), z-y\right\rangle$ is bounded above by a negative quantity, where we have set $A=A(x,z)$.
In fact, since $z$ is on the segment $[y,x]$, the vectors $z-x$ and $z-y$ have opposite directions. That is, there is $\mu<0$ such that $z-y=\mu\,(z-x)$ and so 
$|z-y|=-\mu\,|z-x|$. Then 
\begin{align*}
&\left\langle A (z-x), z-y\right\rangle
=
\mu\,\left\langle A (z-x), z-x\right\rangle\\
&\leq
\lambda\,\mu\,\Phi(x,z)\,|z-x|^2
=
\lambda\,\Phi(x,z)\,\mu\,|z-x|\,|z-x|,\quad \text{from \eqref{eq:ellipticity of A}}\\
&=
-\lambda\,\Phi(x,z)\,|z-y|\,|z-x|
= 
-\lambda\,\Phi(x,z)\,r\,|z-x|.
\end{align*}

If $0< t\leq 1$, then
$1-\dfrac1t\leq 0$ and and once again from \eqref{eq:ellipticity of A}\begin{align*}
0
&\leq
\Delta
\leq
\Lambda \,\Phi(x,z)\,|Tz-z|\,|z-x|
-\dfrac1t\,\lambda\,\Phi(x,z)\,r\,|z-x|
+
\(1-\dfrac1t\)\,
\lambda\,\Phi(x,z)\,|z-x|^2.
\end{align*}
Dividing this inequality by $\Lambda\,\Phi(x,z)$ we obtain
\begin{align*}
0
&\leq
\Delta
\leq
|Tz-z|\,|z-x|
-\dfrac1t\,\dfrac{\lambda}{\Lambda}\,r\,|z-x|
+
\(1-\dfrac1t\)\,
\dfrac{\lambda}{\Lambda}\,|z-x|^2\\
&=
|z-x|\,\(|Tz-z|
-\dfrac1t\,\dfrac{\lambda}{\Lambda}\,r
+
\(1-\dfrac1t\)\,
\dfrac{\lambda}{\Lambda}\,|z-x|\)
\\
&\leq
|z-x|\,\(\epsilon
-\dfrac1t\,\dfrac{\lambda}{\Lambda}\,r
+
\(1-\dfrac1t\)\,
\dfrac{\lambda}{\Lambda}\,|z-x|\)\qquad \text{if $|Tz-z|\leq \epsilon$ from \eqref{eq:main L infinity estimate in a general ball} for $\mathcal E$ small}\\
&\leq
|z-x|\,\(
-\dfrac1t\,\dfrac{\lambda}{2\,\Lambda}\,r
+
\(1-\dfrac1t\)\,
\dfrac{\lambda}{\Lambda}\,|z-x|\)\qquad \text{if $\epsilon\leq \dfrac{\lambda}{2\Lambda}r\(\leq \dfrac{\lambda}{t\,2\,\Lambda}r\)$}\\
&\leq
|z-x|\,\(
-\dfrac1t\,\dfrac{\lambda}{2\,\Lambda}\,r\)\qquad \text{since $1-\dfrac1t\leq 0$.}
\end{align*}
Hence $|z-x|=0$, and therefore $z=x$ obtaining $|x|< \beta_0$, a contradiction.
\end{proof}

We now use this to obtain an estimate for $T^{-1}x-x$, when $T$ is the optimal map for the cost $c(x,y)=h(x-y)$. We have from the theory of optimal transport that 
$
T^{-1}(Tx)=x$ for a.e. $x\in \R^n$.
Then given $0<\beta<1$ we obtain
\begin{align*}
\sup_{y\in B_\beta(0)}|T^{-1}y-y|&=\sup_{T^{-1}(B_\beta(0))}|x-Tx|\\
&\leq
\sup_{B_{\bar \beta}(0))}|x-Tx|\quad \text{from \eqref{eq:inclusion T_t^-1 ball contained into a ball} with $t=1$}\\
&\leq
C\,\(\int_{B_1(0)}|Tx-x|^p\,dx\)^{1/(n+p)}\quad \text{from \eqref{eq:main L infinity estimate in a general ball}}
\end{align*}
for $\mathcal E$ sufficiently small and with $C$ a constant depending only on $ p,n$ and the structural constants of $h$.

\subsection{$L^\infty$-estimates of densities}\label{sec:L-infty-estimates of densities}
We recall that the function $F(A)=\log\(\det A\)$ is concave over the set of matrices $A$ that are positive definite, i.e., $$F\((1-t)\,A+t\,B\)\geq (1-t)\,F(A)+t\,F(B),\quad 0\leq t\leq 1.$$
Exponentiating this yields
\begin{equation}\label{eq:concavity of determinant}
\det\((1-t)\,A+t\,B\)\geq \(\det A\)^{1-t}\(\det B\)^t,\quad 0\leq t\leq 1.
\end{equation}
Let $T$ be a measure preserving map $(\rho_0,\rho_1)$, and let $T_t=t\,T+(1-t)\,Id$ be the interpolating map. Assuming the Jacobian matrix $\nabla T$ is positive definite\footnote{A proof of this may be given along the lines of \cite[Section 5.2, Theorem 5.2.1]{2002-agueh-phdthesis} and \cite[Remark 2.9]{Gutierrez:2007fk}, see also \cite[Theorem 7.28, pp. 272-273]{santambrogio-book} when the differentiability of $c,c^*$ at zero is not assumed. Notice also that if $h$ is homogenous of degree $p$, then $h^*$ is homogenous of degree $q$ with $1/p+1/q=1$.}, we get from \eqref{eq:concavity of determinant} that 
\begin{equation}\label{eq:lower bound of det nabla T_t}
\det \(\nabla T_t\)(x)\geq \(\det \nabla T(x)\)^t.
\end{equation}
Let $\rho_t$ be the measure defined by 
$
\rho_t=\(T_t\)_\# \rho_0,
$
that is, $\rho_t(E)=\int_{\(T_t\)^{-1}(E)} \rho_0(x)\,dx$. Assuming invertibility of the matrices involved, changing variables yields
\[
\int_{\(T_t\)^{-1}(E)} \rho_0(x)\,dx
=
\int_E \rho_0\(\(T_t\)^{-1}z\)\,\dfrac{1}{\det \(\(\nabla T_t\)\(\(T_t\)^{-1}z\)\)}\,dz.
\]
That is, the measure $\rho_t$ has density 
\begin{align}\label{eq:formula for rho(x,t)}
\rho(t,z)&=\rho_0\(\(T_t\)^{-1}z\)\,\dfrac{1}{\det \(\(\nabla T_t\)\(\(T_t\)^{-1}z\)\)}\\
&\leq 
\rho_0\(\(T_t\)^{-1}z\)\,\dfrac{1}{\(\det \(\(\nabla T\)\(\(T_t\)^{-1}z\)\)\)^t}\notag
\end{align}
from \eqref{eq:lower bound of det nabla T_t}.
On the other hand, since $T$ is measure preserving 
\[
\rho_0(x)=\det \(\nabla T(x)\)\,\rho_1(Tx)
\]
which combined with the previous inequality yields
\begin{align*}
\rho(t,z)
&\leq
\rho_0\(\(T_t\)^{-1}z\) \,\(\dfrac{\rho_1\(T\(T_t\)^{-1}z\)}{\rho_0\(\(T_t\)^{-1}z\)}\)^t\\
&=
\rho_0\(\(T_t\)^{-1}z\)^{1-t}\,\rho_1\(T\(T_t\)^{-1}z\)^t.
\end{align*}
From \eqref{eq:main L infinity estimate in a general ball}, $T(B_{r_1}(0))\subset B_{r_2}(0)$ for $0<r_1<r_2<1$, when $\mathcal E=\int_{B_1(0)}|Tx-x|^p\,dx$ is sufficiently small.
And, from \eqref{eq:inclusion T_t^-1 ball contained into a ball}, $T_t^{-1}\(B_\beta(0)\)\subset B_{\bar \beta}(0)$ for some $0<\beta<\bar \beta<1$ uniform for $0\leq t\leq 1$.
Hence $T\(T_t\)^{-1}\(B_\beta(0)\)\subset B_{\beta''}(0)$ for some $0<\beta<\bar \beta<\beta''<1$.
Therefore, assuming that $\rho_0(0)=\rho_1(0)=1$ and $\rho_0,\rho_1$ are H\"older continuous of order $\alpha$, we obtain
\[
\rho_0\(\(T_t\)^{-1}z\)=1+\rho_0\(\(T_t\)^{-1}z\)-1\leq 1+[\rho_0]_{\alpha,1}
\]
and
\[
\rho_1\(T\(T_t\)^{-1}z\)=1+\rho_1\(T\(T_t\)^{-1}z\)-1\leq 1+[\rho_1]_{\alpha,1}
\]
for all $z\in B_\beta(0)$.
Consequently 
\[
\sup_{z\in B_\beta(0)}\rho(t,z)
\leq
\(1+[\rho_0]_{\alpha,1}\)^{1-t}
\,
\(1+[\rho_1]_{\alpha,1}
\)^t;
\]
where $[\rho_i]_{\alpha,1}=\sup_{x,y\in B_1(0),x\neq y}\dfrac{|\rho_i(x)-\rho_i(y)|}{|x-y|^\alpha}$.

\subsection{Connection with fluids}\label{sec:fluids}
The connection between the Monge problem and fluid flows was discovered in 
\cite{2000-benamouandbrenierformula} for quadratic costs. It can be seen that it holds also for general cost functions $h(x-y)$ as above.
Suppose $\rho_i$, $i=1,2$ are given, $v:\R^n\times [0,1]\to \R^n$ is a smooth field, and let $\rho(x,t)$ be a smooth solution of the continuity equation
\[
\partial_t\rho+\text{div}_x \(\rho\,v\)=0 \quad \text{for $(x,t)\in \R^n\times [0,1]$ with $\rho(x,i)=\rho_i(x), i=0,1 $.}
\]
Let $T$ be the optimal map of the Monge problem with cost $h$.
Given the interpolating map $T_tx=t\,Tx+(1-t)\,x$, $0\leq t\leq 1$, consider the field
\[
v(x,t)=\(T-Id\)\(T_t^{-1}x\),
\]
and let $\rho(x,t)$ be solution to the continuity equation above with this $v$.
Define 
\begin{equation}\label{eq:optimal j}
j(x,t)=\rho(x,t)\,\(T-Id\)\(T_t^{-1}x\).
\end{equation}
Then
\begin{align*}
\int_0^1\int_{B_\beta}\dfrac{1}{\rho(x,t)^{p-1}}|j(x,t)|^p\,dxdt
&=
\int_0^1\int_{B_\beta}\left|\(T-Id\)\(T_t^{-1}x\)\right|^p\,\rho(x,t)\,dxdt\\
&=
\int_0^1\int_{T_t^{-1}(B_\beta)}\left|Tz-z\right|^p\,\rho(T_tz,t)\,|\det \nabla T_tz|\,dzdt\\
&=
\int_0^1\int_{T_t^{-1}(B_\beta)}\left|Tz-z\right|^p\,\rho_0(z)\,dzdt\quad \text{from \eqref{eq:formula for rho(x,t)}}\\
&\leq
\int_0^1\int_{B_{\beta'}}\left|Tz-z\right|^p\,\rho_0(z)\,dzdt\quad \text{from \eqref{eq:inclusion T_t^-1 ball contained into a ball} for $\beta<\beta'<1$}
\end{align*}
assuming $\mathcal E=\int_{B_1(0)}|Tx-x|^p\,dx$ is sufficiently small.
Here we have assumed that $\rho_0(1)=1$ and $\rho_0\approx 1$ in $B_1$.
 
On the other hand, if $\beta''<\beta$ it follows from \eqref{eq:main L infinity estimate in a general ball} that
\[
\sup_{|x|\leq \beta''}|T_tx|\leq \beta''+\sup_{|x|\leq \beta''}|Tx-x|
\leq
\beta''+\mathcal E^{\text{power}>0}<\beta,
\]
for $\mathcal E$ sufficiently small and therefore
\[
\int_0^1\int_{B_{\beta''}}\left|Tz-z\right|^p\,\rho_0(z)\,dzdt
\leq
\int_0^1\int_{B_\beta}\dfrac{1}{\rho(x,t)^{p-1}}|j(x,t)|^p\,dxdt
\leq
\int_0^1\int_{B_{\beta'}}\left|Tz-z\right|^p\,\rho_0(z)\,dzdt,
\]
for $j$ in \eqref{eq:optimal j}.

\section{Differentiability of Monotone maps}\label{sec:diff monotone maps}
When $T$ is monotone in the standard sense, the idea used in the proof of Theorem \ref{thm:main Linfty estimate} can be implemented in a simpler way to obtain the following estimates for $T$ minus a general affine function.

\begin{lemma}\label{lm:estimates of T with A}
Let $A\in \R^{n\times n}$, $b\in \R^n$, $T$ a monotone operator, $0<\beta<1$, and $u(x)=Tx-Ax-b$.
Then there are positive constants $C_1,C_2$ depending only on the dimension $n$ such that 
\begin{enumerate}
\item[(a)] for $A\neq 0$ we have
\[
\sup_{y\in B_{\beta R}(x_0)}|u(y)|
\leq
C_1\,\(\|A\|\,R\)^{n/(n+1)}\, \(\fint_{B_R(x_0)}|u(x)|\,dx\)^{1/{(n+1)}}
\] 
if 
\[
\dfrac{1}{R}\,\fint_{B_R(x_0)}|u(x)|\,dx\leq
C_2 \,\|A\|\,\(\dfrac{1-\beta}{2}\)^{n+1};
\]
and
\[
\sup_{y\in B_{\beta\,R}(x_0)}|u(y)|\leq
C_1\,\(\(\dfrac{2}{1-\beta}\)^n\,\fint_{B_R(x_0)}|u(x)|\,dx
+(1-\beta)\,R\,\|A\|\)
\]
if 
\[
\dfrac{1}{R}\,\fint_{B_R(x_0)}|u(x)|\,dx\geq
C_2 \,\|A\|\,\(\dfrac{1-\beta}{2}\)^{n+1}.
\]
\item[(b)]
if $A=0$, then
\[
\sup_{y\in B_{\beta\,R}(x_0)}|u(y)|\leq
C_1\,\(\dfrac{2}{1-\beta}\)^n\,\fint_{B_R(x_0)}|u(x)|\,dx.
\]

\end{enumerate}
\end{lemma}

\begin{proof}
%Let us fix $0<\beta<1$, $A$ is a matrix, $b$ is a vector, and let $u(x):=Tx-Ax-b$.  
By monotonicity of $T$,
 \begin{equation}\label{monotony}(u(x)-u(y))\cdot(x-y)\ge -\langle A(x-y),x-y\rangle,\quad \text{for a.e. $x, y$,}
 \end{equation}
which implies 
\[
f(x):=u(y)\cdot (x-y)\leq u(x)\cdot (x-y)+\langle A(x-y),x-y\rangle.
\]
Let $r>0$ and $z_r\in \R^n$ both to be determined, and consider the ball $B_r(z_r)$.
The function $f$ is harmonic in all space so integrating the last inequality for $x$ over $B_r(z_r)$ and applying the mean value theorem yields
\begin{align*}
u(y)\cdot (z_r-y) 
&\leq 
\fint_{B_r(z_r)}u(x)\cdot (x-y)\,dx+ \fint_{B_r(z_r)} \langle A(x-y),x-y\rangle\,dx\\
&\leq
\fint_{B_r(z_r)}|u(x)|\,|x-y|\,dx+\|A\|\, \fint_{B_r(z_r)} |x-y|^2\,dx\\
&=B+C.
\end{align*}
Fix $x_0$, $R>0$, and pick $r>0$, $z_r=y+r\,\dfrac{u(y)}{|u(y)|}$ such that $B_r(z_r)\subset B_R(x_0)$; $u(y)\neq 0$. If $y\in B_{\beta R}(x_0)$, then the inclusion holds if $r< (1-\beta)\,R/2$.
Also, if $x\in B_r(z_r)$, then $|x-y|\leq 2r$.
Hence
\[
B
\leq \dfrac{2\,r}{\omega_n\,r^n}\int_{B_R(x_0)}|u(x)|\,dx,\qquad
C\leq 4\,\|A\|\,r^2,
\]
and consequently
\[
|u(y)|\leq \dfrac{2}{\omega_n\,r^n}\int_{B_R(x_0)}|u(x)|\,dx 
+4\,\|A\|\,r:=F(r)\qquad \forall y\in B_{\beta R}(x_0);\quad r\leq (1-\beta)\,R/2.
\]
We then obtain
\[
\sup_{y\in B_{\beta R}(x_0)}|u(y)|
\leq
\min \left\{F(r):0<r\leq (1-\beta)\,R/2\right\}:=m.
\]
Suppose $A\neq 0$. Set $\Delta=\dfrac{2}{\omega_n} \int_{B_R(x_0)}|u(x)|\,dx$,
so $F(r)=\dfrac{1}{r^n}\,\Delta +4\,\|A\|\,r$. 
We have $F'(r)=-n\,r^{-(n+1)}\Delta + 4\,\|A\|=0$ for $r=r_0:=\(\dfrac{n\,\Delta}{4\,\|A\|}\)^{1/(n+1)}$.
So 
\begin{align*}
&\min \{F(r):0<r<+\infty\}=F(r_0)\\
&=
\(\dfrac{4\|A\|}{n\,\Delta}\)^{n/(n+1)}\,\Delta +
4\,\|A\|\,\(\dfrac{n\,\Delta}{4\|A\|}\)^{1/(n+1)}\\
&=
C_n
\,\|A\|^{n/(n+1)}
\,\(\int_{B_R(x_0)}|u(x)|\,dx\)^{1/{(n+1)}}.
\end{align*}
If $r_0<\dfrac{1-\beta}{2}\,R$, then $m\leq F(r_0)$ and we obtain
\begin{equation}\label{eq:main estimate of u}
\sup_{y\in B_{\beta R}(x_0)}|u(y)|
\leq
C_n\,\(\|A\|\,R\)^{n/(n+1)}\, \(\fint_{B_R(x_0)}|u(x)|\,dx\)^{1/{(n+1)}}
\end{equation}
when $C_n\,\dfrac{1}{\|A\|\,R}\,\fint_{B_R(x_0)}|u(x)|\,dx\leq
\(\dfrac{1-\beta}{2}\)^{n+1}$; 
in such a case we get 
\[
\sup_{y\in B_{\beta R}(x_0)}|u(y)|\leq
C_n\,(1-\beta)\,\|A\|\,R.
\]
On the other hand, if $\dfrac{1-\beta}{2}R\leq r_0$,  then $m=F\(\dfrac{1-\beta}{2}R\)$
and we get
\[
\sup_{y\in B_{\beta\,R}(x_0)}|u(y)|\leq
C_n\,\(\dfrac{2}{1-\beta}\)^n\,\fint_{B_R(x_0)}|u(x)|\,dx
+C_n\,\,(1-\beta)\,R\,\|A\| 
\]
when
$C_n\,\dfrac{1}{\|A\|\,R}\,\fint_{B_R(x_0)}|u(x)|\,dx\geq
\(\dfrac{1-\beta}{2}\)^{n+1}$.

If $A=0$, then $F(r)=\dfrac{1}{r^n}\,\Delta$ is decreasing 
and so 
\[
\sup_{y\in B_{\beta\,R}(x_0)}|u(y)|\leq
m
=
C_n\,\(\dfrac{2}{1-\beta}\)^n\,\fint_{B_R(x_0)}|u(x)|\,dx.
\]

\end{proof}

Using part (b) of this lemma we will show strong differentiability of monotone maps.
Following Calder\'on and Zygmund \cite{Cal-Z}, see also \cite[Sect. 3.5]{Zim}, we recall the notion of differentiability in $L^p$-sense.
\begin{definition}
Let $1\leq p\leq \infty$, $k$ is a positive integer and $f\in L^p(\Omega)$, with $\Omega\subset \R^n$ open, and let $x_0\in \Omega$.
We say that $f\in T^{k,p}(x_0) \(f\in t^{k,p}(x_0)\)$ if there exists a polynomial $P_{x_0}$ of degree  $\leq k-1 \(\text{$P_{x_0}$ of degree }\leq k\)$ such that 
\begin{align*}
&\(\fint_{B_r(x_0)} |f(x)-P_{x_0}(x)|^p\,dx\)^{1/p}=O(r^k)\quad \text{as $r\to 0$}\\
&\(\(\fint_{B_r(x_0)} |f(x)-P_{x_0}(x)|^p\,dx\)^{1/p}=o(r^k)\quad \text{as $r\to 0$}\);
\end{align*}
when $p=\infty$ the averages are replaced by $\text{\rm ess sup}_{x\in B_r(x_0)}|f(x)-P_{x_0}(x)|=\|f-P_{x_0}\|_{L^\infty \(B_r(x_0)\)}$.
\end{definition}
We mention the following landmark result of Calder\'on and Zygmund \cite[Thm. 5]{Cal-Z}, see also \cite[Thm. 3.8.1]{Zim} or \cite[Chap. VIII, Sect. 6.1]{Stein}:
\begin{theorem}\label{thm:calderon-zygmund}
If $1<p\leq \infty$ and $f\in T^{k,p}(x_0)$ for all $x_0\in E$ with $E\subset \R^n$ measurable, then
$f\in t^{k,p}(x_0)$ for almost all $x_0\in E$; emphasizing that the orders of magnitude are not necessarily uniform in $x_0$\footnote{Whether this result holds when $p=1$ does not seem available in the literature.}.
\end{theorem}

The case when $p=\infty$ is a famous theorem of Stepanov which combined with Lemma \ref{lm:estimates of T with A}(b) yields immediately the following point-wise differentiability of monotone maps.

\begin{theorem}\label{thm:differentiability of monotone maps}
Let $T$ be a monotone map that is locally in $L^1\(\R^n\)$
\footnote{In general, $T$ is a multivalued map. However, the monotonicity implies that $Tx$ is a singleton for a.e. $x\in \R^n$.
Denote $\text{dom }T=\{x\in \R^n: Tx\neq \emptyset \}$. From \cite[Corollary 12.38]{RW:variational analysis}, a maximal monotone mapping $T$ is locally bounded at $\bar x$ if and only if $\bar x$ is not a boundary point of $\overline{\text{dom }T}$. Also from \cite[Thm. 12.63]{RW:variational analysis}, if $T$ is maximal monotone, then $T$ is continuous at $\bar x$ if and only if $T$ is single valued at $\bar x$, in which case necessarily $\bar x\in \text{int }\(\text{dom }T\)$. For a clear and in depth presentation of the properties of monotone maps we recommend the comprehensive book \cite{RW:variational analysis}.} 
satisfying 
\begin{equation}\label{eq:L2-integrability of T}
\fint_{B_R(x_0)}|Tx-b|\,dx=O\(R\)\quad \text{as $R\to 0$} 
\end{equation}
for some vector $b=b_{x_0}$, i.e, $Tx\in T^{1,1}(x_0)$ for all $x_0$ in a measurable set $E$.
Then  
\[
\|Tx-A(x-x_0)-Tx_0\|_{L^\infty\(B_R(x_0)\)}=o\(R\)\quad \text{as $R\to 0$}
\] 
for almost all $x_0\in E$, i.e., $Tx\in t^{1,\infty}(x_0)$ for a.e. $x_0\in E$.
%In other words, if $T$ has derivative of order one in the $L^2$-sense at $x_0$, then $T$ has an ordinary derivative of order one at $x_0$.

\end{theorem} 
\begin{proof}
For each $x_0\in E$ there exist constants $M(x_0)\geq 0$, $R_0>0$ and $b\in \R^n$ such that
\[
\fint_{B_R(x_0)}|Tx-b|\,dx\leq M(x_0)\,R
\]
for all $0<R<R_0$, i.e., $Tx\in T^{1,1}(x_0)$.
Since $T$ is monotone, from Lemma \ref{lm:estimates of T with A}(b)
\[
\sup_{B_{\beta R}(x_0)}|Tx-b|\leq
C(n,\beta)\,\fint_{B_R(x_0)}|Tx-b|\,dx
\leq
C(n,\beta)\,M(x_0)\,R
\]
for $0<R<R_0$. This means $\sup_{B_{R}(x_0)}|Tx-b|=O(R)$ as $R\to 0$ for all $x_0\in E$, i.e., $Tx\in T^{1,\infty}(x_0)$.
By Stepanov's theorem \cite[Chap. VIII, Thm. 3, p. 250]{Stein} this implies that 
$Tx$ is differentiable for a.e. $x_0\in E$, i.e., $Tx\in t^{1,\infty}(x_0)$ for a.e. $x_0\in E$.
\end{proof}

\begin{remark}\label{rmk:integral diff implies T11}\rm
Notice that $\fint_{B_R(x_0)}|Tx-Ax-b|\,dx=o\(R\)$ (or $Tx\in t^{1,1}(x_0)$) implies \eqref{eq:L2-integrability of T}  because
if $x_0$ is a Lebesgue point, then  $b=Tx_0-Ax_0$ and
\begin{align*}
\fint_{B_R(x_0)}|Tx-c|\,dx
&=\fint_{B_R(x_0)}|Tx-Ax-b+Ax+b-c|\,dx\\
&\leq
\fint_{B_R(x_0)}|Tx-Ax-b|\,dx
+
\fint_{B_R(x_0)}|Ax+b-c|\,dx\\
&=
o(R)
+
\fint_{B_R(x_0)}|A(x-x_0)|\,dx,\quad \text{if $c=Tx_0$}\\
&\leq
o(R)+\|A\|\,R=O(R).
\end{align*}
\end{remark}
\begin{remark}\rm\label{rmk:mignot ambrosio and krylov results}
When $T$ is a monotone map that is maximal, the differentiability of $T$ a.e. was proved by Mignot \cite[Thm. 3.1]{Mig} using Sard's Theorem; see also the more recent and perhaps simpler proof of Alberti and Ambrosio \cite[Thm. 3.2]{AA}.
When $T$ is monotone and bounded the differentiability is proved in \cite[Thm. 2.2]{K83}.
\end{remark}

\begin{remark}\rm
If $\phi$ is a convex function in $\R^n$, then from \cite[Thm. 3, p. 240]{evans-gariepy} $\nabla \phi\in BV_{\text{loc}}(\R^n)$. Therefore, from \cite[Thm. 1, p. 228]{evans-gariepy} $\nabla \phi$ is $L^{n/(n-1)}$-differentiable a.e., that is $\nabla \phi\in t^{1,n/(n-1)}(x)$ a.e., and since $\nabla \phi$ is monotone, it follows from Remark \ref{rmk:integral diff implies T11} 
and Theorem \ref{thm:differentiability of monotone maps} that $\nabla \phi\in t^{1,\infty}(x)$ a.e.
\end{remark}

\begin{remark}\rm\label{eq:ambrosio result bounded deformation}
Following \cite{ACDM}, a locally integrable mapping $u:\R^n\to \R^n$ is of bounded deformation ($u\in BD$) if the symmetrized gradient in the sense of distributions $\nabla u+(\nabla u)^t$ is a Radon measure. 
If $T=(T_1,\cdots ,T_n)$ is a monotone map in $L^1_{\text{loc}}(\R^n)$, it then follows from the definitions of monotonicity and distributional derivative that $T\in BD$.
Because all distributional derivatives $\dfrac{\partial T_i}{\partial x_j}$ are non negative and therefore they are Radon measures. 
From \cite[Theorem 7.4]{ACDM}, if $T\in BD$, then $T\in t^{1,1}(x_0)$ for a.e. $x_0\in \R^n$. Therefore from Remark \ref{rmk:integral diff implies T11}, condition \eqref{eq:L2-integrability of T} holds
for any locally integrable monotone map.
\end{remark}
\begin{remark}\rm
For completeness we also prove the following known fact: if $f\in L^p_{\text{loc}}(\R^n)$, with $p\geq 1$, then 
\[
\lim_{r\to 0}\(\fint_{B_r(x_0)}|f(x)-f(x_0)|^p\,dx \)^{1/p}=0\quad \text{for a.e. $x_0$}.
\]
Define 
\[
\Lambda f(x_0)=\limsup_{r\to 0}\(\fint_{B_r(x_0)}|f(x)-f(x_0)|^p\,dx \)^{1/p}.
\]
We have $0\leq
\Lambda f(x_0)\leq
\limsup_{r\to 0}\(\fint_{B_r(x_0)}|f(x)|^p\,dx \)^{1/p}+|f(x_0)|\leq 
\(M(|f|^p)(x_0)\)^{1/p}+|f(x_0)|$ with $M$ the Hardy-Littlewood maximal function. Since $f\in L^p_{\text{loc}}(\R^n)$, the right hand side of the last inequality is finite for a.e. $x_0$ and so $\Lambda f(x_0)$ is finite for a.e. $x_0$.
In addition, $\Lambda$ is sub-linear: $\Lambda (f+g)(x_0)\leq \Lambda f(x_0)+\Lambda g(x_0)$ and $\Lambda g(x_0)=0$ for each $g$ continuous at $x_0$. By localizing $f$ with a compact support function we may assume $f\in L^p(\R^n)$. Given $\e>0$ there exists $g\in C(\R^n)$ such that $\|f-g\|_p\leq \e$.
For each $\a>0$ we then have
\begin{align*}
\{x:\Lambda f(x)>\a\}
&\subset 
\{x:\Lambda (f-g)(x)>\a/2\}\cup \{x:\Lambda g(x)>\a\}
=
\{x:\Lambda (f-g)(x)>\a/2\}\\
&\subset
\{x:\(M(|f-g|^p)(x)\)^{1/p}>\a/4\}\cup \{x:|f(x)-g(x)|>\a/4\}
\end{align*}
and so
\begin{align*}
|\{x:\Lambda f(x)>\a\}|
&\leq
|\{x:M(|f-g|^p)(x)>(\a/4)^p\}|
+
|\{x:|f(x)-g(x)|>\a/4\}|\\
&\leq
\dfrac{C_n}{\a^p}\|f-g\|_p^p+\dfrac{4^p}{\a^p}\|f-g\|_p^p\leq \dfrac{C}{\a^p}\,\e^p.
\end{align*}
Since $\e$ is arbitrary, we obtain $\Lambda f(x)= 0$ for a.e. $x$.
\end{remark}

\section{Appendix}
\setcounter{equation}{0}
Recall that $\Gamma(x)=\dfrac{1}{n\omega_n (2-n)}|x|^{2-n}$, with $n>2$ where $\omega_n$ is the volume of the unit ball in $\R^n$, and  
the Green's representation formula
\[
v(y)=\int_{\partial \Omega} \(v(x)\,\dfrac{\partial \Gamma}{\partial \nu}(x-y)-\Gamma(x-y)\,\dfrac{\partial v}{\partial \nu}(x)\)\,d\sigma(x)+\int_\Omega \Gamma(x-y)\,\Delta v(x)\,dx
\]
where $\nu$ is the outer unit normal and $y\in \Omega$.
If $\Omega=B_\rho(y)$, then 
$\dfrac{\partial \Gamma}{\partial \nu}(x-y)=\dfrac{1}{n\,\omega_n}\,|x-y|^{1-n}$ and so the representation formula reads
\begin{align*}
v(y)
&=
\fint_{|x-y|=\rho}v(x)\,d\sigma(x)-\Gamma(\rho)\,\int_{|x-y|=\rho} \dfrac{\partial v}{\partial \nu}(x)\,d\sigma(x)
+\int_{|x-y|\leq \rho} \Gamma(x-y)\,\Delta v(x)\,dx\\
&=
\fint_{|x-y|=\rho}v(x)\,d\sigma(x)
+\int_{|x-y|\leq \rho} \(\Gamma(x-y)-\Gamma(\rho)\)\,\Delta v(x)\,dx
\end{align*}
from the divergence theorem.
Multiplying the last identity by $\rho^{n-1}$ and integrating over $0\leq \rho\leq r$ yields
\begin{equation}\label{eq:third Green identity}
v(y)=\fint_{|x-y|\leq r}v(x)\,dx+\dfrac{n}{r^n}\,\int_0^r \rho^{n-1}\int_{|x-y|\leq \rho}\(\Gamma(x-y)-\Gamma(\rho)\)\,\Delta v(x)\,dx\,d\rho.
\end{equation}

%\bibliography{monamp.bib}
%\bibliographystyle{alpha}
%\nocite{2000-benamouandbrenierformula}
\end{document}